\documentclass[12pts]{article}
\usepackage[
margin=3cm,
includefoot,
footskip=12pt,
]{geometry}
\usepackage{layout}
\usepackage[utf8]{inputenc}
\usepackage{graphicx}
\usepackage{amsmath}
\usepackage{amsfonts}
\usepackage{amssymb}
\usepackage{amsthm,fullpage}
\usepackage{graphics}
\usepackage{graphicx}
\usepackage{tikz}
\newtheorem{thm}{Theorem}[section]
\newtheorem{lem}{Lemma}[section]

\newtheorem{defn}{Definition}[section]

\newtheorem{rem}{Remark}[section]
\newtheorem{obs}{Observation}[section]
\DeclareMathOperator{\spec}{spec}

\title{On the construction of cospectral graphs for the adjacency and normalized Laplacian matrices}
\author{M. Rajesh Kannan\thanks{Department of Mathematics, Indian Institute of Technology Kharagpur, Kharagpur 721 302, India. Email: rajeshkannan@maths.iitkgp.ac.in, rajeshkannan1.m@gmail.com } \and Shivaramakrishna Pragada\thanks{Undergraduate Student, Department of Aerospace Engineering,  Indian Institute of Technology Kharagpur, Kharagpur 721 302, India. Email: shivaram@iitkgp.ac.in, shivaramkratos@gmail.com }
}
\date{\today}
\begin{document}
    \maketitle
    \baselineskip=0.25in
    \begin{abstract}
In [Steve Butler. A note about cospectral graphs for the adjacency and normalized Laplacian matrices. Linear Multilinear Algebra, 58(3-4):387-390, 2010.], Butler constructed a family of bipartite graphs, which are cospectral for both the adjacency and the normalized Laplacian matrices. In this article, we extend this construction for generating larger classes of bipartite graphs, which are cospectral for both the adjacency and the normalized Laplacian matrices. Also, we provide a couple of constructions of non-bipartite graphs, which are cospectral for the adjacency matrices but not necessarily for the normalized Laplacian matrices.
    \end{abstract}

    {\bf AMS Subject Classification(2010):} 05C50.

    \textbf{Keywords.} Adjacency matrix, Normalized Laplacian matrix, Cospectral graphs.
\section{Introduction} 

Spectral graph theory studies some of the properties of a graph by associating various types of matrices with it. The spectrum of the matrices associated with a graph determines some of the structural properties of that graph. However, not all features are revealed by the spectrum of the associated matrices. One such instance is that two non-isomorphic graphs can have the same spectrum. In this paper, we propose some methods which could be used to construct a pair of non-isomorphic graphs sharing the same spectrum.

All graphs considered in this paper are simple and undirected. Let $G=(V,E)$ be a graph with the vertex set $V=\{1,2,\ldots, n\}$ and the edge set $E$. If two vertices $i$ and $j$ of $G$ are adjacent, we denote it by $i\sim j$. The adjacency matrix $A=[a_{ij}]$ of a graph $G$, on $n$ vertices, is an $n \times n$ matrix defined by
$$a_{ij}=\begin{cases}
1,& if\ i\sim j,\\
0,& otherwise.
\end{cases}.$$ The spectrum, denoted by $\spec(G)$, of $G$ is the set of all eigenvalues of $A(G)$, with corresponding multiplicities. It has been a longstanding problem to characterize graphs that are determined by their spectrum \cite{vandam-haemers, devel-vandam-haemers}. For a graph $G$, if $\spec(H) = \spec(G)$ for any other graph $H$ implies $H$ is isomorphic to $G$, then $G$ is determined by its spectrum, or a DS graph for short, if it is not the case, then $G$ is not determined by its spectrum or NDS graph for short. In \cite{schw}, the author proved that almost all trees are cospectral. In \cite{god-mcka}, the authors provided a method for constructing NDS graphs. That is, constructing non-isomorphic graphs with the same spectrum. In \cite{vandam-haemers}, the authors conjectured that almost all graphs are DS. For more details, we refer to \cite{ brou-haem, cvec, god-mcka, ham-spen, vandam-haemers, devel-vandam-haemers}.

For each vertex $i$ of a graph $G$, let $d_i$ denote the degree of the vertex $i$. Let $D$ denote the diagonal matrix whose $(i,i)$-th entry is $d_i$. Then the matrix $L=D-A$ is called the Laplacian matrix of the graph $G$, and the matrix $\mathcal{L} = D^{-\frac{1}{2}}LD^{-\frac{1}{2}}$ is called the normalized Laplacian matrix of a graph $G$ without isolated vertices. For more details, we refer to \cite{chung1}. In \cite{but-jas}, the authors provided various constructions that give families of graphs which are cospectral for the normalized Laplacian matrices, while, in \cite{but-lama}, the author constructed families of bipartite graphs which are cospectral for both the adjacency and the normalized Laplacian matrices.

In this paper, we generalize a construction given in \cite{but-lama}, which will provide a construction of large classes of NDS bipartite graphs, which are cospectral for both the adjacency and the normalized Laplacian matrices. Thereupon, we extend this construction in a different way, which leads to a larger class of graphs that are cospectral for the adjacency matrices.

This article is organized as follows: In section 2, we collect necessary definitions and known results, which will be used later on in the paper. In section 3, we present the main results.

\section{Definitions, notions and known results }
In this section, we collect some of the known definitions and results. 
\begin{defn}
The direct sum of two matrices $A$ and $B$ is defined to be the block diagonal matrix $\begin{bmatrix}
A      & 0 \\
0       & B
\end{bmatrix},$ where the off-diagonal zero matrices are of appropriate order.  The direct sum of two matrices $A$ and $B$ is denoted by $A \oplus B$.
\end{defn}

\begin{defn}\cite{ben-gre}
For a given $n \times n$ matrix $A$, the unique $n \times n$ matrix satisfying $AXA=A,~XAX=X,~(AX)^{T}=AX$ and
$(XA)^{T}=XA$ is called the \emph{Moore-Penrose inverse} or the \emph{pseudo inverse} of $A$ and is denoted by $A^{\dagger}$.
\end{defn}
In \cite{but-lama}, the author established the following:
\begin{thm} \label{Butler-1}
     Let $p \geq q$, and let $B$ be a $p \times q$ matrix. If $C = \begin{bmatrix}
    0       & B & B \\
    B^T       & 0& 0 \\
    B^T       & 0&  0\\
    \end{bmatrix}$ ( $(2q+p) \times (2q+p)$ matrix) and $D = \begin{bmatrix}
    0       & B^T & B^T\\
    B       & 0& 0 \\
    B       & 0& 0 \\
    \end{bmatrix}$ ($(2p+q) \times (2p+q)$ matrix), then $C \oplus 0_{p-q}$ and $D$ are cospectral.
\end{thm}
\begin{defn}\cite{but-lama}
For a non-negative symmetric matrix $M$ with positive row sums, the normalized Laplacian of the matrix $M$, denoted by $L(M)$,  is defined as $L(M) = I - D^{-\frac{1}{2}}MD^{-\frac{1}{2}}$ where $D$ is the diagonal matrix with diagonal entries composed of the row sums of $M$.\end{defn}

    In case that the non-negative symmetric matrix $M$ has zero row sums, then we use the pseudo inverse of $D$ to define the normalized Laplacian matrix of $M$.

In the case of the adjacency matrix corresponding to a graph, zero row sums correspond to the isolated vertices in the graph. So $A$, the adjacency matrix of the graph with zero row sums,  could be decomposed as $A$ = $A_{1} \oplus 0_{k}$, where $A_{1}$ is the adjacency matrix of all components which do not contain isolated vertices, and $k$ is the number of isolated vertices. In this case, we define the normalized Laplacian as follows: $L(A) =L(A_{1} \oplus 0_{k}) = L(A_{1}) \oplus I_{k}$, where $I_{k}$ denotes the $k \times k$ identity matrix and $0_{k}$ denotes the $k \times k$ zero matrix.

\begin{thm} \cite{but-lama}\label{butler-2} Let $L(M)$ denote the normalized Laplacian of a non-negative symmetric matrix $M$ with positive row sums, then we have the following:
\begin{enumerate}
  \item   A vector $x$ is an eigenvector of $L(M)$ if and only if $y = D^{-\frac{1}{2}}x$ (known as the harmonic eigenvector) satisfies $(D-M)y = \lambda Dy.$
\item If $M$ has $0$ as an eigenvalue with multiplicity $q$, then $L(M)$ has $1$ as an eigenvalue   with multiplicity at least $q$.

\end{enumerate}
\end{thm}

%
%
%

\begin{thm}[{\cite[Lemma 1.1]{but-lama}}]
 Let $p \geq q$, and let $B$ be a $p \times q$ matrix. If $C = \begin{bmatrix}
    0       & B & B \\
    B^T       & 0& 0 \\
    B^T       & 0&  0\\
    \end{bmatrix}$ ( $(2q+p) \times (2q+p)$ matrix) and $D = \begin{bmatrix}
    0       & B^T & B^T\\
    B       & 0& 0 \\
    B       & 0& 0 \\
    \end{bmatrix}$ ($(2p+q) \times (2p+q)$ matrix), then $L(D)$ and $L(C \oplus 0_{p-q})$ are co-spectral. 
\end{thm}

\section{Main results}

The following observation  will be useful for the rest of the article. 
\begin{obs}\label{key-lemma}
    Let  $A = \begin{bmatrix}
    G'  & B  & B\\
    B^T  & 0 & G \\
    B^T  & G & 0 \\
    \end{bmatrix}$,  where $B$ is a $p \times q$ matrix, and $G$ and $G'$ are symmetric matrices of appropriate orders. If $\begin{bmatrix}
        x      \\
        y        \\
        z \\
        \end{bmatrix}$ is an eigenvector of $A$ with corresponding to an eigenvalue $\lambda$, then
        $$\begin{bmatrix}
        G'  & B  & B\\
        B^T  & 0 & G \\
        B^T  & G & 0 \\
        \end{bmatrix}\begin{bmatrix}
        x      \\
        y        \\
        z \\
        \end{bmatrix} =\begin{bmatrix}
        G'x + By+Bz      \\
        B^Tx + Gz        \\
        B^Tx + Gy \\
        \end{bmatrix}=\lambda \begin{bmatrix}
        x      \\
        y        \\
        z \\
        \end{bmatrix}.$$ By comparing the second and third components, we get  $$G(z-y) = -\lambda(z-y).$$  Thus, either $-\lambda$ is an eigenvalue of the matrix $G$ or the vectors  $y$ and $z$ are the same.
  \end{obs}
\subsection{Construction I}
Let $I$ and $0$ denote the identity matrix and the zero matrix of appropriate order, respectively. 
\begin{thm} \label{constr-cospec1.1}
Let $B$ be a $p \times q$ matrix where $p \geq q$. Let $C = \begin{bmatrix}
0       & B & B &\hdots & B\\
B^T       & 0& 0&\hdots  & 0  \\
B^T       & 0& 0&\hdots  & 0  \\
\vdots & \vdots&  \vdots&  \ddots & \vdots\\
B^T       & 0& 0&\hdots  & 0
\end{bmatrix}$ be an  $(nq+p) \times (nq+p)$ matrix, and let $D=  \begin{bmatrix}
0       & B^T & B^T &\hdots & B^T\\
B       & 0& 0&\hdots  & 0  \\
B       & 0& 0&\hdots  & 0  \\
\vdots & \vdots&  \vdots&  \ddots & \vdots\\
B       & 0& 0&\hdots  & 0  \\
\end{bmatrix}$ be an $(np+q) \times (np+q)$ matrix. Then the matrices  $D$ and $C \oplus 0_{(n-1)(p-q)}$ are cospectral.
\end{thm}

\begin{proof}
The matrix $C$ has $0$ as an eigenvalue with multiplicity at least $(n-1)q$.

Let $\lambda$ be an eigenvalue of $\begin{bmatrix}
0      & B \\
B^T       & 0
\end{bmatrix}$, and $\begin{bmatrix}
x \\
y
\end{bmatrix}$ be an eigenvector corresponding to $\lambda$.
Then,  $B^Tx = \lambda y$ and $By = \lambda x$.
By a generalisation of observation \ref{key-lemma} we have, $$\begin{bmatrix}
0       & B & B &\hdots & B\\
B^T       & 0& 0&\hdots  & 0  \\
B^T       & 0& 0&\hdots  & 0  \\
\vdots & \vdots&  \vdots&  \ddots & \vdots\\
B^T       & 0& 0&\hdots  & 0  \\
\end{bmatrix}\begin{bmatrix}
\sqrt n x       \\
y        \\
y \\
\vdots\\
y  \\
\end{bmatrix} = \begin{bmatrix}
 nBy       \\
\sqrt nB^Tx        \\
\sqrt nB^Tx        \\
\vdots \\
\sqrt nB^Tx  \\
\end{bmatrix} = \lambda \sqrt n \begin{bmatrix}
\sqrt n x       \\
y        \\
y \\
\vdots \\
y  \\
\end{bmatrix}. $$
Hence, the  eigenvalues of $C$ are  $\lambda\sqrt n$ and $0$ with multiplicity $(n-1)q$, where $\lambda$ is an eigenvalue of the matrix $\begin{bmatrix}
0      & B \\
B^T       & 0 \\
\end{bmatrix}$.

It is easy to see that, the matrix $D$ has eigenvalue $0$ with multiplicity at least $(n-1)p$, and the  rest of the eigenvalues  of $D$ are $\lambda \sqrt{n}$, where $\lambda$ is an eigenvalue of the matrix $\begin{bmatrix}
0      & B \\
B^T       & 0 \\
\end{bmatrix}.$ For, $$\begin{bmatrix}
0       & B^T & B^T &\hdots & B^T\\
B       & 0& 0&\hdots  & 0  \\
B       & 0& 0&\hdots  & 0  \\
\vdots & \vdots&  \vdots&  \ddots & \vdots\\
B       & 0& 0&\hdots  & 0  \\
\end{bmatrix}\begin{bmatrix}
\sqrt n y       \\
x        \\
x \\
\vdots\\
x  \\
\end{bmatrix} = \begin{bmatrix}
nB^Tx       \\
\sqrt nBy        \\
\sqrt nBy        \\
\vdots \\
\sqrt nBy  \\
\end{bmatrix} = \lambda \sqrt n \begin{bmatrix}
\sqrt n y       \\
x        \\
x \\
\vdots \\
x  \\
\end{bmatrix}. $$

So the matrices $C$ and $D$ have the same set of non-zero eigenvalues (including multiplicity). If $p = q$, then $C$ and $D$ are cospectral. If $p\neq q$ and  $p > q$, then the matrices $D$ and $C \oplus 0_{(n-1)(p-q)}$ are cospectral.
\end{proof}

In the next theorem, we establish that the normalized Laplacians of the matrices $D$ and $C \oplus{0}_{(n-1)(p-q)}$ are cospectral.
\begin{thm} \label{NLC&D}
 Let  $B, C$, and $D$  be  three matrices defined as in Theorem \ref{constr-cospec1.1}. Then the matrices,   $L(C \oplus 0_{(n-1)(p-q)})$ and $L(D)$ are cospectral. 
\end{thm}

\begin{proof}
As the matrices $D$ and $C\oplus 0_{(n-1)(p-q)}$ have $0$ with multiplicity $(n-1)p$,  by Theorem \ref{butler-2}, the spectrum of both the matrices $L(D)$ and $L(C \oplus 0_{(n-1)(p-q)})$ have  $1$ as an eigenvalue with multiplicity $(n-1)p$.
    
    Now, the rest of the eigenvalues are constructed using the eigenvectors of $L\Bigg(\begin{bmatrix}
    0      & B \\
    B^T       & 0
    \end{bmatrix}\Bigg)$. Now, 

    $$ \Bigg(\begin{bmatrix}
    D_{1}      & 0 \\
    0      & D_{2}
    \end{bmatrix} - \begin{bmatrix}
    0      & B \\
    B^T       & 0
    \end{bmatrix}\Bigg)\begin{bmatrix}
    x \\
    y
    \end{bmatrix} = \begin{bmatrix}
    D_{1}      & -B \\
    -B^T      & D_{2}
    \end{bmatrix}\begin{bmatrix}
    x \\
    y
    \end{bmatrix} = \begin{bmatrix}
    D_{1}x - By \\
    D_{2}y - B^Tx
    \end{bmatrix} =  \lambda \begin{bmatrix}
    D_{1}x \\
    D_{2}y
    \end{bmatrix}. $$

    So the rest of the eigenvalues of $L(C)$ are given by
 $$\begin{bmatrix}
    nD_{1}       & -B & -B &\hdots & -B\\
    -B^T       & D_{2}& 0&\hdots  & 0  \\
    -B^T       & 0& D_{2}&\hdots  & 0  \\
    \vdots & \vdots&  \vdots&  \ddots & \vdots\\
    -B^T       & 0& 0&\hdots  & D_{2}  \\
    \end{bmatrix}\begin{bmatrix}
    x       \\
    y        \\
    y \\
    \vdots\\
    y  \\
    \end{bmatrix} = \begin{bmatrix}
    nD_{1}x-nBy       \\
    D_{2}y - B^Tx        \\
    D_{2}y - B^Tx        \\
    \vdots \\
    D_{2}y - B^Tx \\
    \end{bmatrix} = \lambda \begin{bmatrix}
    nD_{1} x       \\
    D_{2}y        \\
    D_{2}y \\
    \vdots \\
    D_{2}y  \\
    \end{bmatrix}.$$
    
  Since  $L(C \oplus 0_{(n-1)(p-q)}) = L(C) \oplus I_{(n-1)(p-q)}$, thus $L(C\oplus 0_{(n-1)(p-q)})$ has $1$ as an eigenvalue with multiplicity $(n-1)p$. Thus the spectrum of $L(C \oplus 0_{(n-1)(p-q)} )$ is $L\Bigg(\begin{bmatrix}
    0      & B \\
    B^T       & 0
    \end{bmatrix}\Bigg) \cup \{\underbrace{1,\dots,1}_\text{$((n-1)p)$-times} \}$.
  The rest of the eigenvalues  of $L(D)$ are given by  $$\begin{bmatrix}
    nD_{2}       & -B^T & -B^T &\hdots & -B^T\\
    -B       & D_{1}& 0& \hdots  & 0  \\
    -B       & 0& D_{1}& \hdots  & 0  \\
    \vdots & \vdots&  \vdots&  \ddots & \vdots\\
    -B       & 0& 0& \hdots  & D_{1}  \\
    \end{bmatrix}\begin{bmatrix}
    y       \\
    x        \\
    x \\
    \vdots\\
    x  \\
    \end{bmatrix} = \begin{bmatrix}
    nD_{2}y-nB^Tx       \\
    D_{1}x - By        \\
    D_{1}x - By        \\
    \vdots \\
    D_{1}x - By \\
    \end{bmatrix} = \lambda \begin{bmatrix}
    nD_{2} y       \\
    D_{1}x        \\
    D_{1}x \\
    \vdots \\
    D_{1}x  \\
    \end{bmatrix}.$$
  Hence, $L(C \oplus 0_{(n-1)(p-q)})$ and $L(D)$ are cospectral.
\end{proof}

\begin{rem}
	In Theorem \ref{constr-cospec1.1}, if the entries of the matrix $B$ are either $0$ or $1$, $p \geq q$ and choose $B$ such that maximum row sum of $B$ is different from maximum column sum of $B$, then the graphs associated with  $D$ and $C \oplus 0_{(n-1)(p-q)}$ as  adjacency matrices, respectively, are cospectral. Also, by Theorem \ref{NLC&D},  the normalized Laplacian matrices corresponding to the adjacency matrices are cospectral. But these graphs not isomorphic.
\end{rem}

Next, we illustrate the above construction with a couple of examples. 
 \begin{enumerate}
        \item For $p = 3$, $q = 2$ and $n = 3$,  let us generate a pair of non-isomorphic, but cospectral graphs with respect to both the adjacency and the normalized Laplacian matrices.
       Let $B =\begin{bmatrix}
        1    &  0 \\
         1      & 1  \\
           1    &   1  \\
        \end{bmatrix}.$
       In the graphs illustrated below, vertex labeled $\{1, 2, \dots, n\}$ corresponds to graph whose  adjacency matrix is $D$, and vertex labeled $\{ 1', 2', \dots, n'\}$  corresponds to graph whose  adjacency matrix is $C \oplus 0_2$.
             
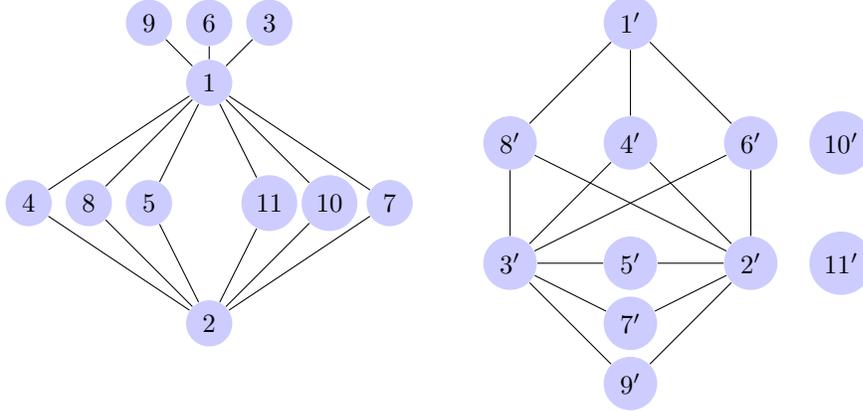
\begin{center}
        	\begin{tikzpicture}
        	[scale=.8,auto=left,every node/.style={circle,fill=blue!20}]
        	\node (n4) at (2,8)  {4};
        	\node (n1) at (5,10) {1};
        	\node (n2) at (5,6)  {2};
        	\node (n3) at (6,11)  {3};
        	\node (n5) at (4,8) {5};
        	\node (n6) at (5,11) {6};
        	\node (n7) at (8,8) {7};
        	\node (n8) at (3,8) {8};
        	\node (n9) at (4,11) {9};
        	\node (n10) at (7,8) {10};
        	\node (n11) at (6,8) {11};
        	
        	\node (n4') at (12,9)  {4$'$};
        	\node (n1') at (12,11) {1$'$};
        	\node (n2') at (14,7)  {2$'$};
        	\node (n3') at (10,7)  {3$'$};
        	\node (n9') at (12,5) {9$'$};
        	\node (n6') at (14,9) {6$'$};
        	\node (n7') at (12,6) {7$'$};
        	\node (n8') at (10,9) {8$'$};
        	\node (n5') at (12,7) {5$'$};
        	\node (n10') at (15.5,9) {10$'$};
        	\node (n11') at (15.5,7) {11$'$};
        	
        	\foreach \from/\to in {n1/n8,n1/n9,n1/n3,n1/n6,n1/n7,n1/n10,n1/n4,n1/n11,n1/n5,n2/n7,n2/n5,n2/n8,n2/n10,n2/n4,n2/n11,n1'/n6',n1'/n8',n1'/n4',n4'/n3',n4'/n2',n6'/n2',n6'/n3',n9'/n2',n9'/n3',n7'/n2',n7'/n3',n5'/n2',n5'/n3',n8'/n2',n8'/n3'}
        	\draw (\from) -- (\to);
        	
        	\end{tikzpicture}
              \end{center}

\item For $p =  q = n = 3$, let us generate a pair of non-isomorphic, but cospectral graphs with respect to both the adjacency and the normalized Laplacian matrices. 
Let  $B = \begin{bmatrix}
1    &  0  & 1\\
1      & 1 & 0 \\
1    &   1 & 0 \\
\end{bmatrix}.$
 In the graphs illustrated below, vertex labeled$\{1, 2, \dots, n\}$ corresponds to graph whose  adjacency matrix is $D$, and vertex labeled$\{ 1', 2', \dots, n'\}$  corresponds to graph whose  adjacency matrix is $C$.

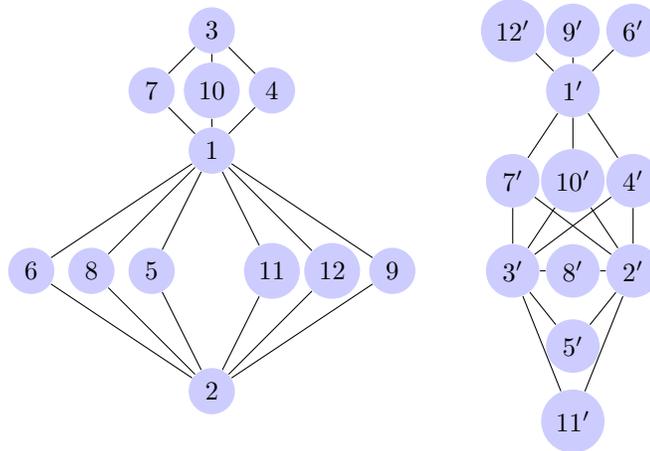
\begin{center}

	\begin{tikzpicture}
	[scale=.8,auto=left,every node/.style={circle,fill=blue!20}]
	\node (n6) at (5,8.5)  {6};
	\node (n1) at (8,10.5) {1};
	\node (n2) at (8,6.5)  {2};
	\node (n3) at (8,12.5)  {3};
	\node (n5) at (7,8.5) {5};
	\node (n10) at (8,11.5) {10};
	\node (n9) at (11,8.5) {9};
	\node (n8) at (6,8.5) {8};
	\node (n7) at (7,11.5) {7};
	\node (n12) at (10,8.5) {12};
	\node (n11) at (9,8.5) {11};
	\node (n4) at (9,11.5) {4};
	
	\node (n6') at (15,12.5)  {6$'$};
	\node (n3') at (13,8.5) {3$'$};
	\node (n2') at (15,8.5)  {2$'$};
	\node (n1') at (14,11.5)  {1$'$};
	\node (n5') at (14,7.25) {5$'$};
	\node (n10') at (14,10) {10$'$};
	\node (n8') at (14,8.5) {8$'$};
	\node (n9') at (14,12.5) {9$'$};
	\node (n7') at (13,10) {7$'$};
	\node (n12') at (13,12.5) {12$'$};
	\node (n11') at (14,6) {11$'$};
	\node (n4') at (15,10) {4$'$};
	
	\foreach \from/\to in {n1/n11,n2/n11,n1/n12,n2/n12,n9/n1,n9/n2,n5/n1,n8/n2,n6/n2,n6/n1,n5/n2,n8/n1,n1/n4,n1/n10,n7/n1,n3/n4,n3/n10,n3/n7,n8'/n3',n8'/n2',n5'/n2',n3'/n5',n11'/n2',n3'/n11',n4'/n2',n4'/n3',n10'/n2',n10'/n3',n7'/n3',n2'/n7',n1'/n10',n1'/n4',n1'/n7',n1'/n9',n1'/n6',n1'/n12'}
	\draw (\from) -- (\to);
	
	\end{tikzpicture}
\end{center}
     \end{enumerate}

\subsection{Generalization of Construction I}

Let $n$ be a positive integer, and  $\sigma(n)$ denote the number of divisors of $n$. Let $B$ be a $p \times q$ matrix, and $k$  be a divisor of $n$. Let us  construct the block matrix $F_k$ of size $kq + \frac{n}{k} p$ as follows: First row of the block matrix $F_k$ consists of $k$ number of $B$'s such that the $B$'s are kept in $(1,2)$th, $(1,3)$th, $\dots$, $(1,k+1)$th  position of $F_k$, and the remaining $\frac{n}{k}$ blocks of $F_k$ are filled with zero matrices of appropriate order. Symmetrically fill the first column of the block matrix $F_k$. Now, if the $i$-th entry of the first column $F_k$ is the zero block, then set the $i$-th row of the matrix $F_k$ equals to the first row of $F_k$. Let us fill the remaining entries symmetrically. In this process, the matrix $B$ is used for $n$ times, and hence, by symmetry, the matrix $B^T$ is too used for $n$ times.  The matrix $F_k$ constructed as above is the following:   $$  \begin{bmatrix}
0        & B &\hdots & B & 0 &\hdots & 0\\
B^T       & 0&\hdots  & 0 & B^T &\hdots & B^T  \\
\vdots &  \vdots&  \ddots & \vdots & \vdots&  \ddots & \vdots \\
B^T      & 0&\hdots  & 0 & B^T &\hdots & B^T \\
0      & B &\hdots & B & 0 &\hdots & 0\\
\vdots &  \vdots&  \ddots & \vdots & \vdots&  \ddots & \vdots \\
0      & B &\hdots & B & 0 &\hdots & 0\\
\end{bmatrix}.$$
It is easy to see that, there are $\sigma(n)$ possibilities are there for the matrix $F_k$.
In the next theorem, we shall show that all these $\sigma(n)$ matrices have all nonzero eigenvalues  in common. 
\begin{thm} \label{num_divisors-theorem}
     Let $B$ be a $p \times q$ matrix. Consider the family of symmetric matrices $F_k$  constructed as above. Then  the matrices $F_k \oplus 0_{(n -\frac{n}{k})p + (1-k)q}$, with $k$ varying over the set of all divisors of $n$, are cospectral.
\end{thm}

\begin{proof}

Let us compute the spectrum of the matrix $F_k$.  

As we have $k$ columns of with first entry as $B$ are the same, so $0$ is an eigenvalue with multiplicity at least $(k-1)q$. Similarly, we have $\frac{n}{k}$ columns of with first entry as $0$ are the same, so $0$ is an eigenvalue with multiplicity at least $(\frac{n}{k} - 1)p$.

The rest of the eigenvalues are obtained by using the eigenvectors of the matrix $\begin{bmatrix}
	0      & B \\
	B^T       & 0
	\end{bmatrix}$, which is illustrated below.
	
Let $\lambda$ be an eigenvalue of $\begin{bmatrix}
0      & B \\
B^T       & 0
\end{bmatrix}$, and $\begin{bmatrix}
x \\
y
\end{bmatrix}$ be an eigenvector corresponding to $\lambda$.
As, $B^Tx = \lambda y$ and $By = \lambda x$, we get
$$ \begin{bmatrix}
0        & B &\hdots & B & 0 &\hdots & 0\\
B^T       & 0&\hdots  & 0 & B^T &\hdots & B^T  \\
\vdots &  \vdots&  \ddots & \vdots & \vdots&  \ddots & \vdots \\
B^T      & 0&\hdots  & 0 & B^T &\hdots & B^T \\
0      & B &\hdots & B & 0 &\hdots & 0\\
\vdots &  \vdots&  \ddots & \vdots & \vdots&  \ddots & \vdots \\
0      & B &\hdots & B & 0 &\hdots & 0\\
\end{bmatrix}\begin{bmatrix}
\sqrt k x       \\
\sqrt\frac{n}{k}y        \\
\vdots\\
\sqrt\frac{n}{k}y  \\
\sqrt k x \\
\vdots\\
\sqrt k x \\
\end{bmatrix} = \begin{bmatrix}
\\k\sqrt\frac{n}{k}By       \\
\frac{n}{k}\sqrt k Bx         \\
\vdots\\
\frac{n}{k}\sqrt k Bx     \\
k\sqrt\frac{n}{k}By\\
\vdots\\
k\sqrt\frac{n}{k}By \\
\end{bmatrix}  = \lambda \sqrt n \begin{bmatrix}
\sqrt k x       \\
\sqrt\frac{n}{k}y        \\
\vdots\\
\sqrt\frac{n}{k}y  \\
\sqrt k x \\
\vdots\\
\sqrt k x \\
\end{bmatrix}.$$
Thus, the rest of the eigenvalues of $F_k$ are  $\lambda \sqrt n$,  where $\lambda$ is an eigenvalue of $\begin{bmatrix}
0      & B \\
B^T       & 0
\end{bmatrix}$.

Hence, the collection of matrices $F_k$ have all non-zero eigenvalues in common and the collection of matrices generated by $F_k \oplus 0_{(n -\frac{n}{k} )p + (1-k)q}$ with $k$ varying over divisors of $n$ are cospectral.
\end{proof}

\begin{thm} \label{NLF_K}
	The normalized Laplacian matrices corresponding to a family of graphs corresponding to  matrices $F_k \oplus 0_{(n -\frac{n}{k})p + (1-k)q}$ constructed in Theorem \ref{num_divisors-theorem} are cospectral.
\end{thm}

\begin{proof}
	As the collection of matrices $F_k \oplus 0_{(n -\frac{n}{k})p + (1-k)q}$ have $0$ as an eigenvalue with multiplicity at least $(n-1)p$,  by Theorem \ref{butler-2}, the spectrum of the family of matrices $L(F_k \oplus 0_{(n -\frac{n}{k})p + (1-k)q})$ have  $1$ as an eigenvalue with multiplicity at least $(n-1)p$.
	
	Now, the  rest of the eigenvalues  are constructed using the eigenvectors of $L\Bigg(\begin{bmatrix}
	0      & B \\
	B^T       & 0
	\end{bmatrix}\Bigg)$.
	
	$$ \Bigg(\begin{bmatrix}
	D_{1}      & 0 \\
	0      & D_{2}
	\end{bmatrix} - \begin{bmatrix}
	0      & B \\
	B^T       & 0
	\end{bmatrix}\Bigg)\begin{bmatrix}
	x \\
	y
	\end{bmatrix} = \begin{bmatrix}
	D_{1}      & -B \\
	-B^T      & D_{2} \\
	\end{bmatrix}\begin{bmatrix}
	x \\
	y \\
	\end{bmatrix} = \begin{bmatrix}
	D_{1}x - By \\
	D_{2}y - B^Tx
	\end{bmatrix} =  \lambda \begin{bmatrix}
	D_{1}x \\
	D_{2}y
	\end{bmatrix}. $$
	
	So the rest of the eigenvalues of $L(F_k)$ are given by
	$$ \begin{bmatrix}
	kD_1        & -B &\hdots & -B & 0 &\hdots & 0\\
	-B^T       & \frac{n}{k} D_2 &\hdots  & 0 & -B^T &\hdots & -B^T  \\
	\vdots &  \vdots&  \ddots & \vdots & \vdots&  \ddots & \vdots \\
	-B^T      & 0&\hdots  & \frac{n}{k} D_2 & -B^T &\hdots & -B^T \\
	0      & -B &\hdots & -B & kD_1  &\hdots & 0\\
	\vdots &  \vdots&  \ddots & \vdots & \vdots&  \ddots & \vdots \\
	0      & -B &\hdots & -B & 0 &\hdots & kD_1 \\
	\end{bmatrix}\begin{bmatrix}
	 x       \\
	y        \\
	\vdots\\
	y  \\
	 x \\
	\vdots\\
	 x \\
	\end{bmatrix} = \begin{bmatrix}
	kD_1x - k By  \\
	\frac{n}{k} D_2y - \frac{n}{k} Bx \\
	\vdots\\
	\frac{n}{k} D_2y - \frac{n}{k} Bx \\
	k D_1x - k By \\
	\vdots\\
	k D_1x - k By  \\
	\end{bmatrix}  = \lambda \begin{bmatrix}
	 k D_1x       \\
	\frac{n}{k} D_2y        \\
	\vdots\\
	\frac{n}{k} D_2y  \\
	k D_1x \\
	\vdots\\
	k D_1x \\
	\end{bmatrix}.$$
	
Since $L(F_k \oplus 0_{(n -\frac{n}{k})p + (1-k)q}) = L(F_k) \oplus I_{(n -\frac{n}{k})p + (1-k)q}$,  thus $L(F_k \oplus 0_{(n-1)(p-q)})$ has $1$ as eigenvalue with multiplicity at least $(n-1)p$.   
	
	Thus spectrum of $L(F_k \oplus 0_{(n -\frac{n}{k})p + (1-k)q})$ is spectrum of $L\Bigg(\begin{bmatrix}
	0      & B \\
	B^T       & 0
	\end{bmatrix}\Bigg) \cup \{ \underbrace{1,\dots,1}_\text{$((n-1)p)$-times} \}$.

	Hence the family of matrices constructed by varying $k$ among the divisors of $n$, the matrices $L(F_k \oplus 0_{(n -\frac{n}{k})p + (1-k)q})$  are cospectral.	
\end{proof}
	
\begin{rem}
	In Theorem \ref{num_divisors-theorem}, if the entries of the matrix $B$ are either $0$ or $1$, $p \geq q$ and if $B$ is such that the maximum row sum of $B$ is different from the maximum column sum of $B$, then the family of graphs associated with $F_k \oplus 0_{(n -\frac{n}{k})p + (1-k)q}$ as adjacency matrices, respectively, are cospectral.  Also, by Theorem \ref{NLF_K},  the normalized Laplacian matrices corresponding to the adjacency matrices are cospectral. But these graphs not isomorphic. In the construction of matrix $F_k$ even if we permute $0$'s and $B$'s, we get the same graph with permuted labels. 
\end{rem}	
 Next, we illustrate the above construction with an example. For $p=3$, $q=2$ and $n=4$. As number of divisors of $4$ is $3$, let us generate $3$ graphs which are cospectral with respect to both the adjacency and the normalized Laplacian matrices, but non-isomorphic. 

Let $B =\begin{bmatrix}
            1    &  0 \\
            1      & 1  \\
            1    &   1  \\
        \end{bmatrix}.$      
In the graphs illustrated below, vertex labeled$\{1, 2, \dots, n\}$ corresponds to the graph of the adjacency matrix $F_1$ with $k = 1$, vertex labeled$\{1', 2', \dots, n'\}$ corresponds to the graph of the adjacency matrix $F_4 \oplus 0_3$ with $k = 4$, and vertex labeled$\{1", 2", \dots, n"\}$ corresponds to the graph of the adjacency matrix $F_2 \oplus 0_4$ with $k = 2$.

        \begin{center}
        	\begin{tikzpicture}
        	[scale=.8,auto=left,every node/.style={circle,fill=blue!20}]
        	\node (n4) at (1,7)  {4};
        	\node (n1) at (4,9) {1};
        	\node (n2) at (4,5)  {2};
        	\node (n3) at (5,10)  {3};
        	\node (n5) at (3,7) {5};
        	\node (n6) at (2,10) {6};
        	\node (n7) at (7,7) {7};
        	\node (n8) at (2,7) {8};
        	\node (n9) at (3,10) {9};
        	\node (n10) at (6,7) {10};
        	\node (n11) at (5,7) {11};
        	\node (n12) at (6,10) {12};
        	\node (n13) at (0,7) {13};
        	\node (n14) at (8,7) {14};
        	
        	\node (n4') at (11,7)  {4$'$};
        	\node (n3') at (14,9) {3$'$};
        	\node (n2') at (14,5)  {2$'$};
        	\node (n14') at (16,10)  {14$'$};
        	\node (n8') at (13,7) {8$'$};
        	\node (n6') at (15,7) {6$'$};
        	\node (n7') at (18,7) {7$'$};
        	\node (n5') at (12,7) {5$'$};
        	\node (n13') at (12,10) {13$'$};
        	\node (n10') at (17,7) {10$'$};
        	\node (n11') at (16,7) {11$'$};
        	\node (n12') at (14,7) {12$'$};
        	\node (n9') at (10,7) {9$'$};
        	\node (n1') at (14,11) {1$'$};
        	    	
        	\foreach \from/\to in {n1/n8,n1/n9,n1/n3,n1/n6,n1/n7,n1/n10,n1/n4,n1/n11,n1/n5,n2/n7,n2/n5,n2/n8,n2/n10,n2/n4,n2/n11,n1/n12,n1/n13,n1/n14,n2/n14,n2/n13,n11'/n3',n11'/n2',n5'/n2',n3'/n5',n9'/n2',n9'/n3',n7'/n3',n7'/n2',n4'/n3',n4'/n2',n8'/n2',n8'/n3',n6'/n2',n6'/n3',n10'/n2',n10'/n3',n1'/n8',n1'/n4',n1'/n6',n1'/n10'}
        	\draw (\from) -- (\to);
        	
        	\end{tikzpicture}
        \end{center}
    
\begin{center}
	\begin{tikzpicture}
	[scale=.8,auto=left,every node/.style={circle,fill=blue!20}]
	\node (n2) at (6,7)  {2$"$};
	\node (n9) at (9,8.5) {9$"$};
	\node (n4) at (9,5.5)  {4$"$};
	\node (n14) at (10.5,10)  {14$"$};
	\node (n5) at (9,10) {5$"$};
	\node (n11) at (6,10) {11$"$};
	\node (n7) at (12,7) {7$"$};
	\node (n8) at (7.5,7) {8$"$};
	\node (n13) at (7.5,10) {13$"$};
	\node (n3) at (10.5,7) {3$"$};
	\node (n10) at (9,4) {10$"$};
	\node (n12) at (12,10) {12$"$};
	\node (n1) at (4.5,7) {1$"$};
	\node (n6) at (13.5,7) {6$"$};
	
	\foreach \from/\to in {n1/n9,n1/n4,n6/n4,n6/n9,n3/n9,n3/n4,n2/n4,n2/n9,n8/n4,n8/n9,n7/n4,n7/n9,n10/n2,n10/n8,n10/n3,n10/n7,n5/n2,n5/n8,n5/n3,n5/n7}
	\draw (\from) -- (\to);
	
	\end{tikzpicture}
\end{center}
\subsection{Construction II}
\begin{thm} \label{G_construct_II}
    Let $B$ be a $p \times q$ matrix, let $G$ be a $q \times q$ symmetric matrix and $G'$ be a $p \times p$ symmetric matrix. Let $ A = \begin{bmatrix}
    G'  & B  & B\\
    B^T  & 0 & G \\
    B^T  & G & 0 \\
    \end{bmatrix}$ be a $(p+2q) \times (p+2q)$ matrix and $ C = \begin{bmatrix}
    G  & B^T  & B^T \\
    B  & 0 & G' \\
    B  & G' & 0 \\
    \end{bmatrix}$ be a $(2p+q) \times (2p+q)$ matrix.
       Then  $ A \oplus \begin{bmatrix}
    0  & G' \\
    G'  & 0 \\
    \end{bmatrix} \oplus G$ and  $ C \oplus \begin{bmatrix}
    0  & G \\
    G  & 0 \\
    \end{bmatrix} \oplus G'$ are cospectral matrices.
\end{thm}
\begin{proof}
If $\lambda$ is an eigenvalue of the matrix $G$ with an eigenvector $y$, then 
     $$\begin{bmatrix}
    G'  & B  & B\\
    B^T  & 0 & G \\
    B^T  & G & 0 \\
    \end{bmatrix}\begin{bmatrix}
    0      \\
    y        \\
    -y \\
    \end{bmatrix} =\begin{bmatrix}
    0      \\
    -Gy        \\
    Gy \\
    \end{bmatrix}=-\lambda \begin{bmatrix}
    0      \\
    y        \\
    -y \\
    \end{bmatrix}.$$
    Thus, $-\lambda$ is an eigenvalue of the matrix  $A$, whenever $\lambda$ is an eigenvalue of $G$. 
    
If $\mu$ is an eigenvalue of the matrix $G'$ with an eigenvector $z$, then 
    $$\begin{bmatrix}
    G  & B^T  & B^T\\
    B  & 0 & G' \\
    B  & G' & 0 \\
    \end{bmatrix}\begin{bmatrix}
    0      \\
    z        \\
    -z \\
    \end{bmatrix} =\begin{bmatrix}
    0      \\
    -G'z        \\
    G'z \\
    \end{bmatrix}=-\mu \begin{bmatrix}
    0      \\
    z        \\
    -z \\
    \end{bmatrix}.$$
       Thus, $-\mu$ is an eigenvalue of the matrix  $C$, whenever $\mu$ is an eigenvalue of $G'$. 
    The rest of the eigenvalues of $A$ are given as follows:
   \begin{equation}\label{thm3.5.1}
   \begin{bmatrix}
   G'  & B  & B\\
   B^T  & 0 & G \\
   B^T  & G & 0 \\
   \end{bmatrix}\begin{bmatrix}
   x      \\
   y        \\
   y \\
   \end{bmatrix} =\begin{bmatrix}
   G'x + 2By      \\
   B^Tx + Gy        \\
   B^Tx + Gy \\
   \end{bmatrix}=\lambda \begin{bmatrix}
   x      \\
   y        \\
   y \\
   \end{bmatrix}.
   \end{equation}

    The rest of the eigenvalues of $C$ are the same as the rest of eigenvalues of  $A$. For, from equation  (\ref{thm3.5.1}), we have  $G'x + 2By = \lambda x $ and $B^Tx + Gy = \lambda y$, and hence
    $$\begin{bmatrix}
    G  & B^T  & B^T\\
    B  & 0 & G' \\
    B  & G' & 0 \\
    \end{bmatrix}\begin{bmatrix}
    y      \\
    \frac{x}{2}        \\
    \frac{x}{2}  \\
    \end{bmatrix} =\begin{bmatrix}
    Gy + B^Tx      \\
    By + G' \frac{x}{2}  \\
    By + G' \frac{x}{2}  \\
    \end{bmatrix}=\lambda \begin{bmatrix}
    y      \\
    \frac{x}{2}  \\
    \frac{x}{2}  \\
    \end{bmatrix}.$$
    Thus the rest of the eigenvalues of the both matrices $A$ and $C$ are the same.
    
It is easy to see that, the matrices $A \oplus \begin{bmatrix}
    0  & G' \\
    G'  & 0 \\
    \end{bmatrix} \oplus G$ and $C \oplus \begin{bmatrix}
    0  & G \\
    G  & 0 \\
    \end{bmatrix} \oplus G'$ are cospectral.
\end{proof}
\begin{rem}
	In Theorem \ref{G_construct_II}, if the entries of the matrix $B$ are either $0$ or $1$, $p \geq q$, if $B$ is such that the maximum row sum of $B$ is different from the maximum column sum of $B$ and if both $G$ and $G'$ are adjacency matrices of some graphs, then the graphs associated with $A \oplus \begin{bmatrix}
	0  & G' \\
	G'  & 0 \\
	\end{bmatrix} \oplus G$ and $C \oplus \begin{bmatrix}
	0  & G \\
	G  & 0 \\
	\end{bmatrix} \oplus G'$ as adjacency matrices, respectively, are cospectral. But these graphs not isomorphic.
\end{rem}

Next, we illustrate the above construction with an example. For $p = 3$ and $q = 2$, let us generate a pair of non-isomorphic but cospectral graphs with respect to adjacency matrix.
Let $B = \begin{bmatrix}
1    &  0 \\
1      & 1  \\
1    &   1  \\
\end{bmatrix}, G = \begin{bmatrix}
0    &  1 \\
1      & 0  \\
\end{bmatrix}$ and $G'=\begin{bmatrix}
0 & 1 & 1 \\
1 & 0 & 1 \\
1 & 1 & 0 \\
\end{bmatrix}.$
The corresponding normalized Laplacian matrices are clearly not cospectral as the  graphs have different number of connected components \cite[Lemma 1.7]{chung1}.
In the following example, the graph whose vertices are labeled by the index set  $\{1, 2, \dots, n\}$ is the graph whose adjacency matrix is given by $C \oplus \begin{bmatrix}
0  & G \\
G  & 0 \\
\end{bmatrix} \oplus G'$, and   the graph whose vertices are labeled by the index set $\{1', 2', \dots, n'\}$ is the graph whose adjacency matrix is given by $A \oplus \begin{bmatrix}
0  & G' \\
G'  & 0 \\
\end{bmatrix} \oplus G$.

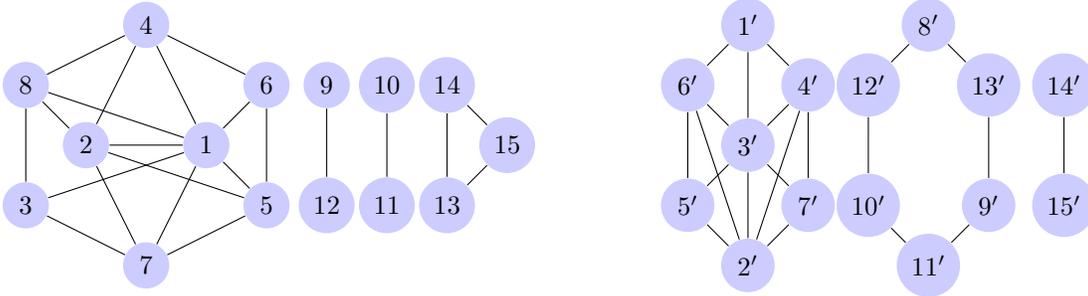
\begin{center}
	\begin{tikzpicture}
	[scale=.8,auto=left,every node/.style={circle,fill=blue!20}]
	\node (n4) at (7,11)  {4};
	\node (n1) at (8,9) {1};
	\node (n2) at (6,9)  {2};
	\node (n3) at (5,8)  {3};
	\node (n5) at (9,8) {5};
	\node (n6) at (9,10) {6};
	\node (n7) at (7,7) {7};
	\node (n8) at (5,10) {8};
	\node (n9) at (10,10) {9};
	\node (n10) at (11,10) {10};
	\node (n11) at (11,8) {11};
	\node (n12) at (10,8) {12};
	\node (n13) at (12,8) {13};
	\node (n14) at (12,10) {14};
	\node (n15) at (13,9) {15};
	
	\node (n6') at (16,10)  {6$'$};
	\node (n1') at (17,11) {1$'$};
	\node (n2') at (17,7)  {2$'$};
	\node (n3') at (17,9)  {3$'$};
	\node (n5') at (16,8) {5$'$};
	\node (n4') at (18,10) {4$'$};
	\node (n7') at (18,8) {7$'$};
	\node (n8') at (20,11) {8$'$};
	\node (n9') at (21,8) {9$'$};
	\node (n10') at (19,8) {10$'$};
	\node (n11') at (20,7) {11$'$};
	\node (n12') at (19,10) {12$'$};
	\node (n13') at (21,10) {13$'$};
	\node (n14') at (22.25,10) {14$'$};
	\node (n15') at (22.25,8) {15$'$};
	
	\foreach \from/\to in {n1/n2,n1/n7,n1/n4,n1/n8,n1/n6,n1/n3,n1/n5,n11/n10,n12/n9,n14/n13,n15/n13,n15/n14,n2/n4,n2/n7,n2/n5,n2/n8,n6/n4,n6/n5,n4/n8,n8/n3,n5/n7,n3/n7,n11'/n9',n9'/n13',n13'/n8',n8'/n12',n12'/n10',n10'/n11',n15'/n14',n1'/n6',n1'/n4',n6'/n5',n5'/n2',n2'/n7',n7'/n4',n3'/n1',n3'/n6',n4'/n3',n3'/n7',n3'/n2',n3'/n5',n2'/n4',n2'/n6'}
	\draw (\from) -- (\to);
	
	\end{tikzpicture}
\end{center}

\subsection{Construction III}
 Let $B$ be a $p \times q$  matrix with $p \geq q$. Let \begin{equation}\label{matrix-c}
 C = \begin{bmatrix}
    0       & B & B &\hdots & B\\
    B^T       & 0& I&\hdots  & I  \\
    B^T       & I& 0&\hdots  & I  \\
    \vdots & \vdots&  \vdots&  \ddots & \vdots\\
    B^T       & I& I&\hdots  & 0  \\
    \end{bmatrix}
 \end{equation}   be an $(nq+p) \times (nq+p)$ matrix, and  \begin{equation}\label{matrix-d}
 D = \begin{bmatrix}
    0       & B^T & B^T &\hdots & B^T\\
    B       & 0& I&\hdots  & I  \\
    B       & I& 0&\hdots  & I  \\
    \vdots & \vdots&  \vdots&  \ddots & \vdots\\
    B       & I& I&\hdots  & 0  \\
    \end{bmatrix}\end{equation} be an $(np+q) \times (np+q)$ matrix.

    \begin{lem}\label{min-one-mul}
    The matrix $C$, defined in equation $(\ref{matrix-c})$, has  $-1$ as an eigenvalue with multiplicity at least $(n-1)q$, and the matrix $D$, defined in equation $(\ref{matrix-d})$, has $-1$  as an eigenvalue with multiplicity   at least $(n-1)p$.
    \end{lem}
    \begin{proof}	
        For any non-zero vector $y$ of size $q$, it is easy to see that $$ \begin{bmatrix}
        0       & B & B &\hdots & B\\
        B^T       & 0& I&\hdots  & I  \\
        B^T       & I& 0&\hdots  & I  \\
        \vdots & \vdots&  \vdots&  \ddots & \vdots\\
        B^T       & I& I&\hdots  & 0  \\
        \end{bmatrix} \begin{bmatrix}
        0       \\
        y        \\
        0 \\
        \vdots\\
        -y  \\
        \end{bmatrix} = -\begin{bmatrix}
        0       \\
        y        \\
        0 \\
        \vdots\\
        -y  \\
        \end{bmatrix}.$$ By placing the vector $y$ from the $2^{nd}$ position to
        the $n^{th}$ position, we get $(n-1)$ linearly independent eigenvectors corresponding to the eigenvalue $-1$. As each $y$ is a $q\times1$ vector,  there are $(n-1)q$ linearly independent eigenvectors associated with the eigenvalue $-1$. Similarly $D$ has $(n-1)p$ linearly independent eigenvectors associated with the eigenvalue $-1$.\end{proof}
        The following lemmas are about the remaining eigenvalues of the matrices $C$ and $D$.
        \begin{lem}\label{lam-imp-nminlam}
        If $\lambda$ is an eigenvalue of $C$ then $(n -\lambda - 1 )$ is an eigenvalue of $D$.
        \end{lem}
        \begin{proof}
        	Let $\lambda$  be an eigenvalue of $C$. Then,
        	\begin{equation}\label{thm3.6.1}
        	\begin{bmatrix}
        	0       & B & B &\hdots & B\\
        	B^T       & 0& I&\hdots  & I  \\
        	B^T       & I& 0&\hdots  & I  \\
        	\vdots & \vdots&  \vdots&  \ddots & \vdots\\
        	B^T       & I& I&\hdots  & 0  \\
        	\end{bmatrix} \begin{bmatrix}
        	x       \\
        	y        \\
        	y \\
        	\vdots\\
        	y  \\
        	\end{bmatrix} = \lambda \begin{bmatrix}
        	x       \\
        	y        \\
        	y \\
        	\vdots\\
        	y  \\
        	\end{bmatrix}.
        	\end{equation}
            As, from equation (\ref{thm3.6.1}), we have $B^Tx + (n-1)y = \lambda y$ and $nBy = \lambda x$ and hence
            $$ \begin{bmatrix}
                0       & B^T & B^T &\hdots & B^T\\
                B       & 0& I&\hdots  & I  \\
                B       & I& 0&\hdots  & I  \\
                \vdots & \vdots&  \vdots&  \ddots & \vdots\\
                B       & I& I&\hdots  & 0  \\
            \end{bmatrix} \begin{bmatrix}
            y       \\
            -\frac{x}{n}        \\
            -\frac{x}{n} \\
            \vdots\\
            -\frac{x}{n}  \\
            \end{bmatrix} = (n-1-\lambda) \begin{bmatrix}
            y       \\
            -\frac{x}{n}        \\
            -\frac{x}{n} \\
            \vdots\\
            -\frac{x}{n}  \\
            \end{bmatrix}. $$
        \end{proof}
    
    \begin{lem}\label{lam-imp-nminlam-ext}
        If $\lambda$ is an eigenvalue of $C$ other than $0$, then $(n -\lambda - 1)$ is an eigenvalue of $C$.
    \end{lem}
\begin{proof}
    Let $\lambda$  be an eigenvalue of $C$. Then, 
    	\begin{equation}\label{thm3.6.2}
        \begin{bmatrix}
            0       & B & B &\hdots & B\\
            B^T       & 0& I&\hdots  & I  \\
            B^T       & I& 0&\hdots  & I  \\
            \vdots & \vdots&  \vdots&  \ddots & \vdots\\
            B^T       & I& I&\hdots  & 0  \\
            \end{bmatrix} \begin{bmatrix}
            x       \\
            y        \\
            y \\
            \vdots\\
            y  \\
            \end{bmatrix} = \lambda \begin{bmatrix}
            x       \\
            y        \\
            y \\
            \vdots\\
            y  \\
            \end{bmatrix}.
            \end{equation}
            
             For, from equation (\ref{thm3.6.2}), we have $B^Tx + (n-1)y = \lambda y$ and $nBy = \lambda x$ and hence $$\begin{bmatrix}
    0       & B & B &\hdots & B\\
    B^T       & 0& I&\hdots  & I  \\
    B^T       & I& 0&\hdots  & I  \\
    \vdots & \vdots&  \vdots&  \ddots & \vdots\\
    B^T       & I& I&\hdots  & 0  \\
    \end{bmatrix} \begin{bmatrix}
    x       \\
    \frac{(n-1-\lambda)}{\lambda}y        \\
    \frac{(n-1-\lambda)}{\lambda}y \\
    \vdots\\
    \frac{(n-1-\lambda)}{\lambda}y  \\
    \end{bmatrix} = (n-1-\lambda) \begin{bmatrix}
    x       \\
    \frac{(n-1-\lambda)}{\lambda}y        \\
    \frac{(n-1-\lambda)}{\lambda}y \\
    \vdots\\
    \frac{(n-1-\lambda)}{\lambda}y  \\
    \end{bmatrix}.$$
\end{proof}
\begin{lem}
    If $\lambda$ is an eigenvalue of $D$ other than $n-1$, then $n -\lambda - 1$ is an eigenvalue of $D$.
\end{lem}
\begin{proof}
The proof is similar to that of Lemma \ref{lam-imp-nminlam-ext}.
\end{proof}

\begin{lem}
Let $B^T$ have full row rank. Then, the matrix $C$ has $0$ as an eigenvalue with multiplicity at least $p-q$, and the matrix $D$ has $n-1$ as an   eigenvalue with multiplicity at least $p-q$.
\end{lem}
\begin{proof}
Since $B^T$ has full row rank, so the dimension of the null space of $B^T$ is $p-q$. Now,  for a vector $x$ in the null space of $B^T$
    $$\begin{bmatrix}
0       & B & B &\hdots & B\\
B^T       & 0& I&\hdots  & I  \\
B^T       & I& 0&\hdots  & I  \\
\vdots & \vdots&  \vdots&  \ddots & \vdots\\
B^T       & I& I&\hdots  & 0  \\
\end{bmatrix} \begin{bmatrix}
x       \\
0        \\
0 \\
\vdots\\
0  \\
\end{bmatrix} = 0. $$ Thus $C$ has $0$ as an eigenvalue with multiplicity at least $(p-q)$, and hence, by Lemma \ref{lam-imp-nminlam}, $D$ has $(n-1)$ as an eigenvalue with multiplicity at least $(p-q)$.
\end{proof}

Now, let us compute the rest of the eigenvalues of $C$. Consider
$$\begin{bmatrix}
0       & B & B &\hdots & B\\
B^T       & 0& I&\hdots  & I  \\
B^T       & I& 0&\hdots  & I  \\
\vdots & \vdots&  \vdots&  \ddots & \vdots\\
B^T       & I& I&\hdots  & 0  \\
\end{bmatrix} \begin{bmatrix}
x       \\
y        \\
y \\
\vdots\\
y  \\
\end{bmatrix} = \lambda \begin{bmatrix}
x       \\
y        \\
y \\
\vdots\\
y  \\
\end{bmatrix}. $$ Then  $B^Tx + (n-1)y = \lambda y$ and $nBy = \lambda x.$ It is easy to see that, the eigenvalue equations of the $2 \times 2$ block matrix  $\begin{bmatrix}
    0  & nB    \\
    B^T & (n-1)I  \\
\end{bmatrix}$  are $B^Tx + (n-1)y = \lambda y$ and $nBy = \lambda x.$

By applying the Schur complement formula with respect to $(1,1)$-th block, we get $$\det\Big(\begin{bmatrix}
\lambda I  & -nB    \\
-B^T & \lambda I - (n-1)I  \\
\end{bmatrix}\Big) = \lambda^{(p-q)} \det(\lambda( \lambda - n +1)I - nB^TB ).$$
Thus giving the rest of the eigenvalues of the matrix $C$. (As we assumed that $\lambda$ is not zero, and since $B$ has full row rank, $B^TB$ is a full rank matrix thus it doesn't have zero eigenvalues, justifying the usage of Schur complement formula.)

Now for the  matrix $D$ similar analysis can be done, and it is easy to show that rest of the eigenvalues of the matrix $D$ corresponds to the block matrix $\begin{bmatrix}
0  & nB^T   \\
B & (n-1)I  \\
\end{bmatrix}.$ 

By the Schur complement formula with respect to $(2,2)$-th block, we get $$\det\Big(\begin{bmatrix}
\lambda I  & -nB^T    \\
-B & \lambda I - (n-1)I  \\
\end{bmatrix}\Big) = (n-1-\lambda)^{(p-q)} \det(\lambda( \lambda - n +1)I - nB^TB ).$$
This gives the rest of the eigenvalues of the matrix $D$.
   \begin{thm}\label{Construct_III}
       The matrices $D \oplus 0_{(p-q)}$ and $  C  \oplus  \underbrace{  (J-I)_n\oplus \cdots \,\oplus (J-I)_n}_\text{$(p-q)$-times} $ are cospectral.
    \end{thm}
\begin{proof}
Follows from the previous lemmas.
\end{proof}

\subsection{Generalization of Construction III}

    Next we provide an alternate proof of Theorem \ref{Construct_III} without full row rank assumption. . The advantage of the previous construction is an insight into the eigenvalues and the structure of the eigenvectors, which is not transparent in this construction. Let $\mathbb{C}(\lambda)$ denote the field of rational functions.

    By Lemma \ref{min-one-mul}, the matrix $C$ has $-1$ as an eigenvalue with multiplicity at least $(n-1)q$, and the matrix $D$ has $-1$ as an eigenvalue with multiplicity at least $(n-1)p$. The rest of the  eigenvalues of the matrix $C$ are given by
    $$\begin{bmatrix}
    0       & B & B &\hdots & B\\
    B^T       & 0& I&\hdots  & I  \\
    B^T       & I& 0&\hdots  & I  \\
    \vdots & \vdots&  \vdots&  \ddots & \vdots\\
    B^T       & I& I&\hdots  & 0  \\
    \end{bmatrix} \begin{bmatrix}
    x       \\
    y        \\
    y \\
    \vdots\\
    y  \\
    \end{bmatrix} =  \begin{bmatrix}
    nBy       \\
    B^Tx + (n-1)y        \\
    B^Tx + (n-1)y \\
    \vdots\\
    B^Tx + (n-1)y  \\
    \end{bmatrix} = \lambda \begin{bmatrix}
    x       \\
    y        \\
    y \\
    \vdots\\
    y  \\
    \end{bmatrix}.$$ The eigenvalue equations of the above matrix lead to that of the $2 \times 2$ block matrix $\begin{bmatrix}
    0  & nB    \\
    B^T & (n-1)I  \\
    \end{bmatrix}$. By the Schur complement formula (in $\mathbb{C}[\lambda]$) with respect to the $(1,1)$-th block, we obtain $$\det\Big(\begin{bmatrix}
    \lambda I  & -nB    \\
    -B^T & \lambda I - (n-1)I  \\
    \end{bmatrix}\Big) = \lambda^{p} \det((\lambda - n +1)I - nB^T(\lambda I)^{-1}B ) = \lambda^{p-q} \det(\lambda(\lambda - n +1)I - nB^TB).$$
As $p \geq q$, we get a polynomial at the end of the above computation, which is the indeed characteristic polynomial of the given matrix.
    Thus giving the rest of the eigenvalues of the matrix $C$. 

    Now, by the same argument, the rest of the eigenvalues of $D$ correspond to the $2 \times 2$ block matrix $\begin{bmatrix}
    0  & nB^T   \\
    B & (n-1)I  \\
    \end{bmatrix}$. By the  Schur complement formula (in $\mathbb{C}(\lambda)$)  with respect to $(2,2)$-th block, we get $$\det\Big(\begin{bmatrix}
    \lambda I  & -nB^T    \\
    -B & \lambda I - (n-1)I  \\
    \end{bmatrix}\Big) = (n-\lambda-1)^{p} \det(\lambda I - nB^T(\lambda I - (n-1)I)^{-1}B ) = (n-\lambda-1)^{(p-q)} \det(\lambda( \lambda - n +1)I - nB^TB ).$$
   As $p \geq q$, we get a polynomial at the end of the above computation, which is the characteristic polynomial of given matrix.
    This gives the rest of the eigenvalues of the matrix $D$.
    
It is easy to see that the matrices $D \oplus 0_{(p-q)}$  and $C \oplus \underbrace{(J-I)_n\oplus \cdots \,\oplus (J-I)_n}_\text{$(p-q)$-times}$ are cospectral.
    
    \begin{rem}
    	In Theorem \ref{Construct_III}, if the entries of the matrix $B$ are either $0$ or $1$, $p \geq q$ and choose $B$ such that the maximum row sum of $B$ is different from the maximum column sum of $B$, then the graphs associated with  $D \oplus 0_{(p-q)}$ and $C \oplus \underbrace{(J-I)_n\oplus \cdots \,\oplus(J-I)_n}_\text{$(p-q)$-times}$ as adjacency matrices, respectively, are cospectral. But these graphs not isomorphic.
    \end{rem}

Next, we illustrate the above construction with two examples. 
\begin{enumerate}
 \item For $p = 3$, $q = 2$ and $n = 2$, let us generate a pair of non-isomorphic but cospectral graphs with respect to the adjacency matrix, but not with respect to the normalized Laplacian matrix.
	Let $B = \begin{bmatrix}
	1    &  1 \\
	1      & 1 \\
	1      & 1 \\
	\end{bmatrix}.$	In the following example, the graph whose vertices are labeled by the index set  $\{1', 2', \dots, n'\}$ is the graph whose adjacency matrix is given by $D \oplus 0_1$, and the graph whose vertices are labeled by the index set $\{1, 2, \dots, n\}$ is the graph whose adjacency matrix is given by $  C  \oplus (J-I)_2$. But the spectrum of matrix $L(D \oplus 0_1)$ is $\{0, 0.6667, 0.6667, 1, 1, 1.3333, 1.3333, 1.3333, 1.6667\}$, and the spectrum of matrix $L(C \oplus (J-I)_2)$ is	$\{0, 0, 0.75, 1, 1, 1.25, 1.25, 1.75, 2 \}.$	
	
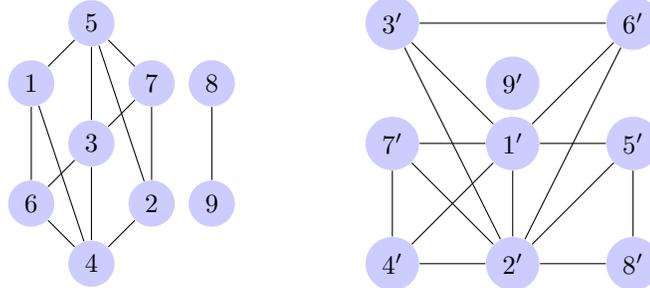
	\begin{center}
		\begin{tikzpicture}
		[scale=.8,auto=left,every node/.style={circle,fill=blue!20}]
		\node (n5) at (7,9) {5};
		\node (n4) at (7,5)  {4};
		\node (n3) at (7,7)  {3};
		\node (n7) at (8,8) {7};
		\node (n6) at (6,6) {6};
		\node (n1) at (6,8)  {1};
		\node (n2) at (8,6) {2};
		\node (n8) at (9,8) {8};
		\node (n9) at (9,6) {9};
		\node (n3') at (12,9)  {3$'$};
		\node (n1') at (14,7) {1$'$};
		\node (n2') at (14,5)  {2$'$};
		\node (n4') at (12,5)  {4$'$};
		\node (n5') at (16,7) {5$'$};
		\node (n6') at (16,9) {6$'$};
		\node (n9') at (14,8) {9$'$};
		\node (n8') at (16,5) {8$'$};
		\node (n7') at (12,7) {7$'$};
		
		\foreach \from/\to in {n8/n9,n5/n1,n1/n6,n6/n4,n4/n2,n2/n7,n7/n5,n3/n7,n3/n6,n5/n3,n3/n4,n1/n4,n5/n2,n1'/n2',n2'/n4',n4'/n7',n7'/n1',n1'/n6',n6'/n3',n3'/n2',n1'/n5',n5'/n8',n8'/n2',n6'/n2',n3'/n1',n1'/n4',n2'/n5',n2'/n7'}
		\draw (\from) -- (\to);
		
		\end{tikzpicture}
	\end{center}	
	\item For $p = 3$, $q = 3$ and $n = 2$, let us generate a pair of non-isomorphic but cospectral graphs with respect to the adjacency matrix but not the normalized Laplacian matrix. 
	Let $B=\begin{bmatrix}
	1    &  0  & 1\\
	1      & 1 & 0 \\
	1    &   1 & 0 \\
	\end{bmatrix}.$ In the graphs illustrated below, vertex labeled by the $\{1', 2', \dots, n'\}$ corresponds to graph of adjacency matrix $D$ and vertex labeled by the $\{1, 2, \dots, n\}$ corresponds to graph of adjacency matrix $C$. The spectrum of the matrix $L(D)$ is $\{0, 0.2034, 0.6738, 1, 1.25, 1.3333, 1.3478, 1.5, 1.6917 \}$, and the spectrum of the matrix $L(C)$ is $\{0, 0.2324, 0.6667, 1, 1.3333, 1.3333,	1.3333, 1.4343,	1.6667\}$.
	
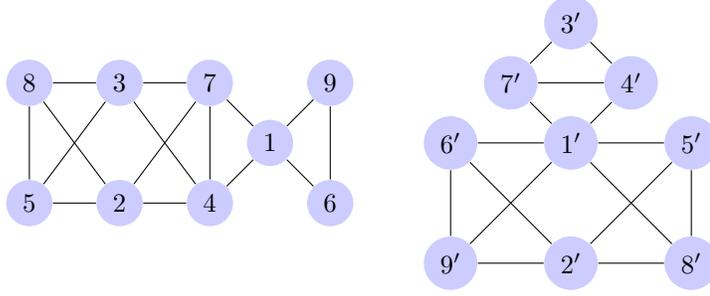
	\begin{center}
		\begin{tikzpicture}
		[scale=.8,auto=left,every node/.style={circle,fill=blue!20}]
		\node (n3) at (9,9)  {3$'$};
		\node (n1) at (9,7) {1$'$};
		\node (n2) at (9,5)  {2$'$};
		\node (n4) at (10,8)  {4$'$};
		\node (n5) at (11,7) {5$'$};
		\node (n9) at (7,5) {9$'$};
		\node (n6) at (7,7) {6$'$};
		\node (n8) at (11,5) {8$'$};
		\node (n7) at (8,8) {7$'$};
		\node (n6') at (5,6) {6};
		\node (n4') at (3,6)  {4};
		\node (n3') at (1.5,8)  {3};
		\node (n7') at (3,8) {7};
		\node (n5') at (0,6) {5};
		\node (n8') at (0,8)  {8};
		\node (n2') at (1.5,6) {2};
		\node (n1') at (4,7) {1};
		\node (n9') at (5,8) {9};
		
		\foreach \from/\to in {n1/n4,n4/n3,n3/n7,n7/n1,n4/n7,n1/n6,n6/n9,n9/n2,n2/n8,n8/n5,n5/n1,n1/n9,n1/n8,n2/n6,n2/n5,n5'/n8',n8'/n3',n3'/n7',n7'/n4',n4'/n2',n2'/n5',n5'/n3',n8'/n2',n3'/n4',n2'/n7',n1'/n9',n1'/n4',n1'/n7',n1'/n6',n9'/n6'}
		\draw (\from) -- (\to);
		
		\end{tikzpicture}
	\end{center} 
\end{enumerate}

    \subsection{Construction IV}
    Let $B$ be a $p \times q$ matrix, $G$ be a $q \times q$ symmetric matrix, and $G'$ be a $p \times p$ symmetric matrix. 
    Let $ A = \begin{bmatrix}
    G'  & B  & B\\
    B^T  & 0 & G \\
    B^T  & G & 0 \\
    \end{bmatrix}$ be a $(p + 2q) \times (p + 2q)$ matrix, and let $ C = \begin{bmatrix}
    G'  & B  & B \\
    B^T  & G & 0 \\
    B^T  & 0 & G \\
    \end{bmatrix}$ be a $(p + 2q) \times (p + 2q)$ matrix. 
    \begin{thm} \label{G_construct_IV}
    The matrices $A \oplus \begin{bmatrix}
    G  & 0 \\
    0  & G \\
    \end{bmatrix}$  and $C \oplus \begin{bmatrix}
    0  & G \\
    G  & 0 \\
    \end{bmatrix}$ are cospectral.
    \end{thm}
    \begin{proof}
    	If $\lambda$ is an eigenvalue of the matrix $G$ with an eigenvector $y$, then 
    	$$\begin{bmatrix}
    	G'  & B  & B\\
    	B^T  & 0 & G \\
    	B^T  & G & 0 \\
    	\end{bmatrix}\begin{bmatrix}
    	0      \\
    	y        \\
    	-y \\
    	\end{bmatrix} =\begin{bmatrix}
    	0      \\
    	-Gy        \\
    	Gy \\
    	\end{bmatrix}=-\lambda \begin{bmatrix}
    	0      \\
    	y        \\
    	-y \\
    	\end{bmatrix}.$$
    	Thus, $-\lambda$ is an eigenvalue of the matrix  $A$, whenever $\lambda$ is an eigenvalue of $G$. 
    	
    	If $\mu$ is an eigenvalue of the matrix $G$ with an eigenvector $z$, then 
    	$$\begin{bmatrix}
    	G'  & B  & B\\
    	B^T  & G & 0 \\
    	B^T  & 0 & G \\
    	\end{bmatrix}\begin{bmatrix}
    	0      \\
    	z        \\
    	-z \\
    	\end{bmatrix} =\begin{bmatrix}
    	0      \\
    	Gz        \\
    	-Gz \\
    	\end{bmatrix}=\mu \begin{bmatrix}
    	0      \\
    	z        \\
    	-z \\
    	\end{bmatrix}.$$
    	Thus, $\mu$ is an eigenvalue of the matrix  $C$, whenever $\mu$ is an eigenvalue of $G$. 
    	The rest of the eigenvalues of $A$ are given as follows:
    	\begin{equation}\label{thm3.7.1}
    	\begin{bmatrix}
    	G'  & B  & B\\
    	B^T  & 0 & G \\
    	B^T  & G & 0 \\
    	\end{bmatrix}\begin{bmatrix}
    	x      \\
    	y        \\
    	y \\
    	\end{bmatrix} =\begin{bmatrix}
    	G'x + 2By      \\
    	B^Tx + Gy        \\
    	B^Tx + Gy \\
    	\end{bmatrix}=\lambda \begin{bmatrix}
    	x      \\
    	y        \\
    	y \\
    	\end{bmatrix}.
    	\end{equation}
The rest of the eigenvalues of $C$ are the same as the rest of the eigenvalues of  $A$. For, from equation  (\ref{thm3.7.1}), we have  $G'x + 2By = \lambda x $ and $B^Tx + Gy = \lambda y$, and hence
    	$$\begin{bmatrix}
    	G'  & B  & B\\
    	B^T  & G & 0 \\
    	B^T  & 0 & G \\
    	\end{bmatrix}\begin{bmatrix}
    	x      \\
    	y       \\
    	y  \\
    	\end{bmatrix} =\begin{bmatrix}
    	G'x + 2By      \\
        B^Tx + Gy  \\
    	B^Tx + Gy  \\
    	\end{bmatrix}=\lambda \begin{bmatrix}
    	x      \\
    	y  \\
    	y  \\
    	\end{bmatrix}.$$
    	Thus the rest of the eigenvalues of the both matrices $A$ and $C$ are same.
    	
    	It is easy to see that, the matrices $A \oplus \begin{bmatrix}
    	G  & 0 \\
    	0  & G \\
    	\end{bmatrix}$  and $C \oplus \begin{bmatrix}
    	0  & G \\
    	G  & 0 \\
    	\end{bmatrix} $ are cospectral.
        \end{proof}

          \subsection{Generalization of Construction IV}
         Let $B$ be a $p \times q$ matrix, $G$ be a $q \times q$ symmetric matrix, and $G'$ be a $p \times p$ symmetric matrix. 
         Let $ A = \begin{bmatrix}
         G'  & B  & B\\
         B^T  & 0 & G \\
         B^T  & G & 0 \\
         \end{bmatrix}$ be a $(p + 2q) \times (p + 2q)$ matrix.
         Let $E$ and $F$ be $q \times q$ symmetric matrices with the property $E + F = G$.  Define $D = \begin{bmatrix}
          G'  & B  & B\\
          B^T  & E & F \\
          B^T  & F & E \\
          \end{bmatrix}$ be a $(p + 2q) \times (p + 2q)$ matrix. 
        \begin{thm} \label{G_construct_IV_g}
            The matrices $A \oplus \begin{bmatrix}
            E  & F \\
            F  & E \\
            \end{bmatrix}$  and $D \oplus \begin{bmatrix}
            0  & G \\
            G  & 0 \\
            \end{bmatrix}$ are cospectral.
        \end{thm}

    \begin{proof}
    	If $\lambda$ is an eigenvalue of the matrix $G$ with an eigenvector $y$, then 
    	$$\begin{bmatrix}
    	G'  & B  & B\\
    	B^T  & 0 & G \\
    	B^T  & G & 0 \\
    	\end{bmatrix}\begin{bmatrix}
    	0      \\
    	y        \\
    	-y \\
    	\end{bmatrix} =\begin{bmatrix}
    	0      \\
    	-Gy        \\
    	Gy \\
    	\end{bmatrix}=-\lambda \begin{bmatrix}
    	0      \\
    	y        \\
    	-y \\
    	\end{bmatrix}.$$
    	Thus, $-\lambda$ is an eigenvalue of the matrix  $A$, whenever $\lambda$ is an eigenvalue of $G$. 
    	
    	If $\mu$ is an eigenvalue of the matrix $E-F$ with an eigenvector $z$, then 
    	$$\begin{bmatrix}
    	G'  & B  & B\\
    	B^T  & E & F \\
    	B^T  & F & E \\
    	\end{bmatrix}\begin{bmatrix}
    	0      \\
    	z        \\
    	-z \\
    	\end{bmatrix} =\begin{bmatrix}
    	0      \\
    	(E-F)z        \\
    	-(E-F)z \\
    	\end{bmatrix}=\mu \begin{bmatrix}
    	0      \\
    	z        \\
    	-z \\
    	\end{bmatrix}.$$
    	Thus, $\mu$ is an eigenvalue of the matrix  $D$, whenever $\mu$ is an eigenvalue of $E-F$. 
    	The rest of the eigenvalues of $A$ are given as follows:
    	\begin{equation} \label{thm3.8.1}
    	\begin{bmatrix}
    	G'  & B  & B\\
    	B^T  & 0 & G \\
    	B^T  & G & 0 \\
    	\end{bmatrix}\begin{bmatrix}
    	x      \\
    	y        \\
    	y \\
    	\end{bmatrix} =\begin{bmatrix}
    	G'x + 2By      \\
    	B^Tx + Gy        \\
    	B^Tx + Gy \\
    	\end{bmatrix}=\lambda \begin{bmatrix}
    	x      \\
    	y        \\
    	y \\
    	\end{bmatrix}.
    	\end{equation}
    	
    	The rest of the eigenvalues of $D$ are the same as the rest of  the eigenvalues of  $A$. For, from equation  (\ref{thm3.8.1}), we have  $G'x + 2By = \lambda x $ and $B^Tx + Gy = \lambda y$, and hence
    	$$\begin{bmatrix}
    	G'  & B  & B\\
    	B^T  & E & F \\
    	B^T  & F & E \\
    	\end{bmatrix}\begin{bmatrix}
    	x      \\
    	y        \\
    	y \\
    	\end{bmatrix} =\begin{bmatrix}
    	G'x + 2By      \\
    	B^Tx + (E+F)y        \\
    	B^Tx + (E+F)y \\
    	\end{bmatrix}=\begin{bmatrix}
    	G'x + 2By      \\
    	B^Tx + Gy        \\
    	B^Tx + Gy \\
    	\end{bmatrix}= \lambda \begin{bmatrix}
    	x      \\
    	y        \\
    	y \\
    	\end{bmatrix}.$$
    	Thus rest of the eigenvalues of the both matrices $A$ and $D$ are the same.
    	
    	It is easy to see that, the matrices $A \oplus \begin{bmatrix}
    	E  & F \\
    	F  & E \\
    	\end{bmatrix}$  and $D \oplus \begin{bmatrix}
    	0  & G \\
    	G  & 0 \\
    	\end{bmatrix} $ are cospectral.
    \end{proof}

    \begin{rem}
    	In Theorem \ref{G_construct_IV}, if the entries of the matrix $B$ are either $0$ or $1$, if $B$ is such that the maximum row sum of $B$ is different from the maximum column sum of $B$ and if both $G$ and $G'$ are adjacency matrices of some simple undirected graphs, then the graphs associated with $A \oplus \begin{bmatrix}
    	G  & 0 \\
    	0  & G \\
    	\end{bmatrix}$  and $C \oplus \begin{bmatrix}
    	0  & G \\
    	G  & 0 \\
    	\end{bmatrix}$ as adjacency matrices, respectively, are cospectral. But these graphs not isomorphic.
    	
    	In Theorem \ref{G_construct_IV_g}, if the entries of the matrix $B$ are either $0$ or $1$, if $B$ is such that maximum row sum of $B$ is different from maximum column sum of $B$ and if both $G$ and $G'$ are adjacency matrices of some simple undirected graphs, and if both $E$ and $F$ have entries as either $0$ or $1$, along with $E$ having zero diagonal entries, then graphs associated with $A \oplus \begin{bmatrix}
    	E  & F \\
    	F  & E \\
    	\end{bmatrix}$  and $D \oplus \begin{bmatrix}
    	0  & G \\
    	G  & 0 \\
    	\end{bmatrix}$ as adjacency matrices, respectively, are cospectral. But these graphs not isomorphic.
    \end{rem}

Next, we illustrate the above construction with an example. For $p = 2$, $q = 3$, let us generate a pair of non-isomorphic but cospectral graphs with respect to adjacency matrix.
Let $ B = \begin{bmatrix}
1    &  1 & 1\\
1      & 1 & 1 \\
\end{bmatrix}$ , $G' = \begin{bmatrix}
0    &  1 \\
1      & 0  \\
\end{bmatrix}$ and $ G = \begin{bmatrix}
0 & 1 & 1 \\
1 & 0 & 1 \\
1 & 1 & 0 \\
\end{bmatrix}.$ The corresponding normalized Laplacian matrices are  not cospectral as the graphs have different number of connected components. In the graphs illustrated below, vertices labeled by the $\{1, 2, \dots, n\}$ corresponds to graph whose  adjacency matrix is $C \oplus \begin{bmatrix}
0  & G \\
G  & 0 \\
\end{bmatrix}$, and vertices labeled by the $\{ 1', 2', \dots, n'\}$  corresponds to graph whose adjacency matrix is $A \oplus \begin{bmatrix}
G  & 0 \\
0  & G \\
\end{bmatrix}$ .
        
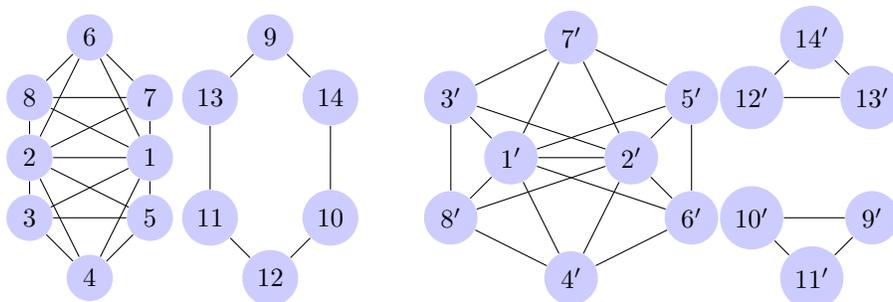
\begin{center}
	\begin{tikzpicture}
	[scale=.8,auto=left,every node/.style={circle,fill=blue!20}]
	\node (n8) at (6,8)  {8};
	\node (n6) at (7,9) {6};
	\node (n1) at (8,7)  {1};
	\node (n2) at (6,7)  {2};
	\node (n3) at (6,6) {3};
	\node (n7) at (8,8) {7};
	\node (n5) at (8,6) {5};
	\node (n9) at (10,9) {9};
	\node (n10) at (11,6) {10};
	\node (n11) at (9,6) {11};
	\node (n12) at (10,5) {12};
	\node (n13) at (9,8) {13};
	\node (n14) at (11,8) {14};
	\node (n4) at (7,5) {4};
	\node (n3') at (13,8)  {3$'$};
	\node (n7') at (15,9) {7$'$};
	\node (n4') at (15,5)  {4$'$};
	\node (n1') at (14,7)  {1$'$};
	\node (n8') at (13,6) {8$'$};
	\node (n5') at (17,8) {5$'$};
	\node (n6') at (17,6) {6$'$};
	\node (n14') at (19,9) {14$'$};
	\node (n9') at (20,6) {9$'$};
	\node (n10') at (18,6) {10$'$};
	\node (n11') at (19,5) {11$'$};
	\node (n12') at (18,8) {12$'$};
	\node (n13') at (20,8) {13$'$};
	\node (n2') at (16,7) {2$'$};
	
	\foreach \from/\to in {n9/n14,n14/n10,n10/n12,n12/n11,n11/n13,n13/n9,n2/n1,n2/n8,n8/n6,n7/n6,n7/n1,n2/n3,n3/n4,n4/n5,n5/n1,n3/n5,n3/n1,n1/n4,n2/n4,n2/n5,n1/n8,n1/n6,n8/n7,n2/n6,n2/n7,n14'/n12',n14'/n13',n13'/n12',n9'/n11',n11'/n10',n10'/n9',n7'/n3',n3'/n8',n8'/n4',n4'/n6',n6'/n5',n5'/n7',n1'/n2',n1'/n7',n1'/n5',n1'/n8',n1'/n4',n1'/n6',n1'/n3',n2'/n7',n2'/n5',n2'/n8',n2'/n4',n2'/n6',n2'/n3'}
	\draw (\from) -- (\to);
	
	\end{tikzpicture}
\end{center}

\textbf{Acknowledgment:}  M. Rajesh Kannan would like to thank the Department of Science and Technology, India, for financial support through the projects MATRICS (MTR/2018/000986) and Early Career Research Award (ECR/2017/000643) .
\bibliographystyle{plain}
\bibliography{raj-shiv}
\end{document}